\documentclass[10pt, reqno]{amsart}
\usepackage{amsmath, amsthm, amscd, amsfonts, amssymb, graphicx, color}
\usepackage{epstopdf}
\usepackage[bookmarksnumbered, colorlinks, plainpages]{hyperref}
\usepackage{graphicx,cite,tikz}
\usepackage{enumerate}%

\usetikzlibrary{shapes.geometric}
\hypersetup{colorlinks=true,linkcolor=red, anchorcolor=green, citecolor=cyan, urlcolor=red, filecolor=magenta, pdftoolbar=true}
\newtheorem{theorem}{Theorem}[section]
\newtheorem{lemma}[theorem]{Lemma}

\newtheorem{corollary}[theorem]{Corollary}
\theoremstyle{definition}
\newtheorem{remark}[theorem]{Remark}
\newtheorem{definition}[theorem]{Definition}
\newtheorem{example}[theorem]{Example}

\theoremstyle{remark}
\numberwithin{equation}{section}

\begin{document}
	\begin{center}
		{\large\textbf{Relation-Theoretic Banach Contraction Principle in Topological Spaces with an Application}}
	\end{center}
	\bigskip
	\begin{center}
		\textbf{Md Hasanuzzaman$^{1}$, Abhishikta Das$^{2}$, Sumit Som$^{3}$}
	\end{center}
	\begin{center}
		
		$^{1}$ Department of Mathematics, Thapar Institute of Engineering and Technology, Patiala Punjab, India\\
        $^{2}$ Department of Mathematics, Siksha-Bhavana, Visva-Bharati, Santiniketan, Birbhum, West Bengal, India \\
		$^{3}$ Department of Mathematics, Dinabandhu Mahavidyalaya, Bongaon, Pin-743235, West Bengal, India\\
		$^{1}$ md.hasanuzzaman1@gmail.com, $^{2}$ abhishikta.math@gmail.com, $^{3}$ somkakdwip@gmail.com
		\bigskip
		
		\textbf{Abstract}
	\end{center}
In this article, we extend several relation-theoretic notions to topological spaces. We introduce relation preserving contraction mapping into topological spaces and utilize the same to extend Banach contraction principle in  topological spaces employing a binary relation. 
 To illustrate the validity of our main result, we provide a concrete example along with a MATLAB-based visualization of the convergence behavior.
Furthermore, we demonstrated the applicability of our main result by finding a solution for a fractional differential equation under some suitable assumptions.\\


\section{Introduction}

Metric fixed point theory is a cornerstone of nonlinear analysis having profound applications in diverse fields. Banach's contraction principle(abbreviated as BCP) \cite{banach1922operations} is a fundamental result in this theory that provides criterion for the existence as well as uniqueness of fixed point for self mappings. The strength of BCP lies in its wide applications which fall in several domain, namely: Differential equation, Integral equation, Economics, fractal theory, aquatic problem, market equilibrium, etc. which leads us to consider the BCP serves as a quintessential example of classical results encompassing all existing fixed point theorems.
Over the years, many generalizations and extensions of this principle have been developed by \\ 
 $\bullet$ relaxing contractive condition of involved map \\
$\bullet$ the number of mappings involved \\
$\bullet$ expanding the underlying space such as b-metric spaces, partial metric spaces, topological spaces $etc.$.  

The study of fixed point results in the setting of topological spaces is of significant interest. It allows the exploration of mappings where metric structures naturally may not exist. In such generalized structures, we need to define contraction mappings, study convergence properties and extend the classical fixed-point results into broader and more general frameworks. Due to this, several auxiliary tools such as: binary relations and continuous functions, have been used to study several important properties of the underlying space. 

In 2015, Alam and Imdad \cite{alam2015relation} introduced the relation-theoretic variant of BCP that unifies several results such as transitive relation due to Turinici \cite{turinici1986fixed}, order-theoretic relation by Ran and Reurings \cite{ran2004fixed}, Nieto and Rodr{\'i}guez-L{\'o}pez \cite{Nieto}, and several others. In this regrad, the technical details are available in Alam and Imdad \cite{alam2015relation} and Alam et al. \cite{alam2022relation}. 

On the other side, Som et al. \cite{S.Som} introduced the notion of topologically BCP  on a    topological space $ \Omega $   and studied the existence of fixed points of such mapping.   This generalization replaced the metric with a continuous function $  h : \Omega \times \Omega \to \mathbb{R} $ that fulfills certain specific conditions.     The interplay between the function $ g $  and the topology of the space allows for the development of novel results applicable to mappings beyond the traditional  framework of metric space. 

 Being inspired by  these foundations, in this article we utilize the concept of  binary relation $ \mathcal{R} $  into the topological structure of the space to establish   fixed-point results that   extend, sharpened versions of some known results of the existing literature. This framework generalizes contraction mappings by defining `topologically $  \mathcal{R} $-preserving contraction' that involves   the relational structure $  \mathcal{R} $.    
 With the interplay of    the relation $   \mathcal{R}  $,    a continuous    function $g$   and the topology of the space,  we establish sufficient conditions to guarantee the existence and uniqueness of fixed points.
  Furthermore,  there are scenarios where the BCP in topological space by S. Som et al. \cite{S.Som} may fail to guarantee the existence of a fixed point. The presence of additional relational structures on the underlying topological space  can make the existence theorem more efficient.   Our article  provides a  counterexample where the classical or topological Banach contraction principle is not applicable.  
  
 The paper is organized as follows. In Section 3, we begin with the formal definition of topologically $  \mathcal{R} $-preserving contractions    followed by the statement and proof of the main result. Example are provide to  highlight the significance of our result. This example   demonstrates that Theorem 2.9 of \cite{S.Som} fails to guarantee a fixed point for a self mapping wherein  utilizing a suitable binary relation $ \mathcal{R}$,  our newly introduced theorem ensures  the existence of a fixed point. The validity of our main result is demonstrated through a concrete example and an effective visualization of the convergence is presented using MATLAB.
 In Section 4, an   application is presented to illustrate the utility and its potential in mathematical analysis of the proposed result in the considered framework.

\section{Relation-theoretic notions}

We wish to recall the following terminological and notational conventions to make our paper possibly self-sustained. In what follows, $\mathbb{N}$, $\mathbb{Q}$  and $\mathbb{R}$ stands for the sets of natural, rational    and real numbers  and  $\mathbb{N}_{0}=\mathbb{N}\cup\{0\}$.\\
\noindent In this continuation, we also summarize some basic   definitions, concepts and relevant auxiliary results as described below:

 A binary relation $\mathcal{R}$ on a non-empty set $ \Omega $ is defined as an arbitrary subset  of $\Omega \times \Omega $. 
 From now on by $\mathcal{R}$, we denote a non-empty binary relation. If $( r, s ) \in \mathcal{R}$ and $ ( s, w ) \in \mathcal{R}$ imply $ ( r, w ) \in \mathcal{R},$ for all $ r, s, w \in \Omega $ then $\mathcal{R}$ is said to be transitive relation on $\Omega $. Furthermore, if $T$ is a self mapping on $ \Omega,$ then $\mathcal{R}$ is said to be $T$-transitive if it is transitive on $ T( \Omega )$.

\begin{definition}\cite{alam2015relation}
	For a binary relation $\mathcal{R}$    on $ \Omega $  
	\begin{enumerate}[(i)]
		\item  inverse relation $\mathcal{R}^{-1}:=\lbrace ( r, s )\in  \Omega^{2} : ( s, r ) \in \mathcal{R} \rbrace $ and  symmetric closure
		$\mathcal{R}^{s} : = \mathcal{R} \cup \mathcal{R}^{-1} $,
		\item $ r $ and $ s $ are $\mathcal{R}$-comparative   if either $( r, s )\in \mathcal{R}$ or $( s, r )\in \mathcal{R}$. It is denoted by $ [ r, s ]  \in \mathcal{R}.$
		\item if $( r, s ) \in \mathcal{R}^{s} \Longleftrightarrow [ r, s ] \in \mathcal{R}$.
		\item  a sequence $\lbrace { r_{n}}\rbrace \subset \Omega $ is called   $\mathcal{R}$-preserving  if
		 $$( r_{n}, r_{n+1})\in \mathcal{R} \hspace{0.5cm} \forall ~~n\in \mathbb{N}_{0}.$$ 
	\end{enumerate}
\end{definition}

\begin{definition} \cite{alam2015relation} \label{fclosed}
	For a  a self-mapping $T$ on nonempty set $ \Omega $, any binary relation $\mathcal{R}$ on $\Omega $ is said to be $T$-closed if for all $ r, s \in \Omega $,
	\begin{center}
		$( r, s ) \in \mathcal{R} \implies (T r, T s) \in \mathcal{R}.$
	\end{center}
\end{definition}

\begin{definition}\cite{kolman2000discrete}
	For $ r, s \in \Omega $, a path from $ r$ to $ s $ having length $n$, $n\in\mathbb{N}$    is a finite sequence  $\{r_0, r_1, r_2, \ldots, r_n\}\subseteq \Omega $ such that $ r_0 = r,\ r_n = s $ with $ ( r_i, r_{i+1} ) \in\mathcal{R},$ for each $i\in\{0,1,\ldots,n-1\}$.
\end{definition}
	
It is worth mentioning here that a path of length $n$ involves $n+1$ elements of $ \Omega $ (not necessarily distinct). 

\begin{definition}\cite{alam0000relationfilomat} 
	A subset $D\subseteq \Omega $ is called $\mathcal{R}$-connected  if, for every $ r, s \in D $, there exists a path in $\mathcal{R}$  connecting $ r $ to $ s $. 
\end{definition}

As we are intended to  use the concept of binary relation $ \mathcal{R} $ into
the topological structure of the space to establish fixed-point results,   the following  concepts of g-convergence and g-completeness by Som et al. \cite{S.Som} are necessary to recall.

\begin{definition}\cite{S.Som}
     Let $ \Omega $ be a topological space,  $ \{ \mu_n\} \subseteq  \Omega $   and $ g : \Omega \times \Omega \to \mathbb{R} $ be a continuous function.   Then 
     \begin{enumerate} [(i)]
         \item $ \{ \mu_n\} $  is called $  g $-convergent to some  $ \mu \in \Omega $ if $ \underset{n\to \infty}{\lim} ~ | g ( \mu_n, \mu ) | = 0 $.
        \item $ \{\mu_n\} $  is said to be $  g $-Cauchy if $ \underset{m, n\to \infty}{\lim} ~ | g ( \mu_n, \mu_m ) | = 0 $.
        \item if every $  g $-Cauchy sequence   in  $ \Omega $  is $  g $-convergent to some  point  in  $ \Omega $ then $ \Omega $ is said to be $  g $-complete.
     \end{enumerate}
\end{definition}
 
\begin{lemma} \label{uniqueness of limit} \cite{S.Som}
	Let $ \Omega $ be a topological space and  $ g: \Omega \times \Omega \to\mathbb{R} $ be a continuous function     satisfying $ g ( r, s) = 0 \implies r = s $ and $ | g(r, s) | \leq | g(r, t) |+ | g(t, s) | $  for all $  r, s, t \in \Omega $.  Then the limit of a $ g $-convergent sequence is unique.
\end{lemma}

\begin{theorem} \cite{S.Som}  \label{thm of S.Som} 
Let $ \Omega $ be a  topological space. Consider a continuous function   $ g : \Omega \times \Omega \to \mathbb{R} $    that satisfies $ g(r, s) = 0 \implies r = s $, $ |g(r, s)| = |g( s, r )| $, and $ |g( r, s )| \leq |g(r, t )| + |g(t, s)| $ for all $ r, s, t \in \Omega $. If $ U : \Omega \to \Omega $ is a topologically Banach contraction mapping with respect to $ g $ and $ \Omega $ is $ g $-complete, then $ U $ has exactly one   fixed point and for any $ \eta_0 \in \Omega $, the sequence $ \{ \eta_{n+1} \} = \{ U(\eta_n) \}  $   converges to the  fixed point of $ U $.
\end{theorem}

In this continuation, we introduce the following definitions   that extend the notions of convergence of sequence, continuity and completeness  in the context of topological spaces equipped with a binary relation and a continuous function.  

\begin{definition}  
Consider a topological space $  \Omega $ endowed with a binary relation $\mathcal{R}$ and  $g: \Omega \times \Omega \to\mathbb{R}$  be a continuous function. Then,
	\begin{enumerate}[(i)]
		\item $ S : \Omega \rightarrow \Omega $ is called $g$-$\mathcal{R}$-continuous at $ r \in \Omega $ if for any $\mathcal{R}$-preserving $g$-convergent sequence $\{ r_{n}\}$ that converges  to $ r $, we have $ S ( r_{n} ) $ is  $g$-convergent to $ S( r ).$ Furthermore, $S $ is said to be $g$-$\mathcal{R}$-continuous if  it satisfies $g$-$\mathcal{R}$-continuity at every point of $ \Omega $.  
		\item If for a $g$-$\mathcal{R}$-convergent sequence $\{ r_{n} \} $ that converges to  $r $, there exists a subsequence $\{ r_{n_{l}}\}$ of $\{r_{n}\}$ with $ ( r_{n_{l}}, r )  \in \mathcal{R}$ ~~for all ~~ $ l \in \mathbb{N}_{0}$, then $\mathcal{R}$ is said to be $g$-self-closed.
         \item  if every $\mathcal{R}$-preserving $g$-Cauchy sequence in $\Omega$ is $g$-convergent then  $ \Omega $ is said to be $g$-$\mathcal{R}$-complete. 
	\end{enumerate}
\end{definition}


\section{Main results}

\noindent We employ the following notations on a topological space $\Omega$ endowed with a binary relation $\mathcal{R}$ and $ S $ a self-mapping  on   $\Omega$:
\begin{enumerate}[(i)]
	\item   $\Omega(S;\mathcal{R}):=\{ u \in \Omega :( u, S u )\in \mathcal{R}\}$,
	\item   $\Upsilon ( u, v, \mathcal{R} ) $:  the class of all paths in $ \mathcal{R} $  from $  u ~  \text{to}~ v ,$
	\item   $ F( S ): \text{set of all fixed points}.$	
\end{enumerate}	

Throughout  the article   $ \Omega $ stands for a topological space. 
Now, we define the concept of topologically $\mathcal{R}$-preserving contraction mapping in   $ \Omega $ with respect to a special function $g$ as follows:

\begin{definition} 
	Suppose $ \Omega $ is endowed with a binary relation $\mathcal{R}$ and $g: \Omega \times \Omega \to\mathbb{R}$  be a continuous function. Then $ S : \Omega \to \Omega $ is said to be topologically $\mathcal{R}$-preserving contraction with respect to $g$ if there exists $\alpha\in ( 0, 1 ) $ such that
 \begin{equation*}
 	|g ( S \mu_1, S \mu_2 ) | \leq \alpha | g ( \mu_1, \mu_2 ) |
 \end{equation*} 
for all $ \mu_1, \mu_2 \in \Omega $ with $ ( \mu_1, \mu_2 ) \in \mathcal{R} $.
\end{definition}

Next, we present and demonstrate our main result.

\begin{theorem}\label{main thm}
	 Let $\mathcal{R}$ be a binary relation on   $ \Omega $.  Suppose $ g  $  is a real-valued continuous	function on $ \Omega \times \Omega $ satisfying 
	 \begin{enumerate}[(g1)]
	 	\item $ g ( r, u ) = 0 \implies r = u $ 
	 	\item $ | g ( r, u ) | = | g ( u, r ) | $ 
	 	\item $ |g( r, u )| \leq |g(r, t ) | + | g ( t, u ) | $
	 \end{enumerate}  
     
 for all $ r, u, t  \in \Omega $ such that $ ( r, u ) \in \mathcal{R} ~ \& ~ ( t, u ) \in \mathcal{R} $   and $ S : \Omega \to \Omega $ is a mapping satisfying the followings:  
	\begin{enumerate}[(i)]
		\item $ \Omega$ is $g$-$\mathcal{R}$-complete,
		\item $\mathcal{R}$ is $ S $-closed,
		\item $ \Omega ( S;\mathcal{R})$ is non-empty,
		\item either $ S$ is ``$g-\mathcal{R}$-continuous" or ``$\mathcal{R}$ is $g$-self-closed",
		\item $ S $ is topologically $\mathcal{R}$-preserving contraction with respect to $g$.\label{iii} 
	\end{enumerate} 
    
	Then $ F (  S ) \neq \Phi $. Moreover, for each $r_0\in \Omega ( S ;\mathcal{R})$, the Picard sequence $ \{ S^n(r_0) \} $   converges to a fixed point of $S$.     
\end{theorem}
\begin{proof}
	Choose $ r_0 \in \Omega ( S; \mathcal{R})$ arbitrarily and construct a sequence $ \{ r_n \} \subset \Omega $ by $$ r_{n+1}=S ( r_n)=\ldots= S ^{n+1}(r_0),\ \text{for all}\ n\in\mathbb{N}_0.$$ 
	
	Now, as $( r_0, S r_0)\in\mathcal{R}$, then due to the $S $-closedness of $\mathcal{R}$, we iteratively get 
	\begin{equation*}
		( S ^{n}(r_0), S^{n+1}(r_0))\in\mathcal{R}\ \text{for all}\ n\in\mathbb{N}_0.
	\end{equation*}
    
	i.e.,	
	\begin{equation}
	(r_n, r_{n+1})\in\mathcal{R}\ \text{for all}\ n\in\mathbb{N}_0.
	\end{equation}
    
	Thus, $ \{ r_n \}$ is $\mathcal{R}$-preserving sequence in $ \Omega $. Now, as $ S $ is topologically $\mathcal{R}$-preserving contraction with respect to $g$, we get	
	\begin{align*}
	  |g(r_n,r_{n+1})| =   |g( S r_{n-1}, S r_{n})| 
	   \leq    \alpha |g(r_{n-1},r_{n})| \leq
	   \ldots    
	   \leq  \alpha^n |g(r_{0},r_{1})|
	\end{align*}
    
	for all $n\in\mathbb{N}_0$. 
    
	Now, for all $m,n\in\mathbb{N}$ with $m<n$, we obtain
		\begin{align*}
		|g(r_m,r_n)| & \leq  |g(r_m,r_{m+1})|+|g(r_{m+1},r_{m+2})|+\ldots+|g(r_{n-1},r_{n})| \\
		& \leq   (\alpha^m+\alpha^{m+1}+\ldots+\alpha^{n-1}) \, |g(r_{0},r_{1})|\\
		& \leq  \alpha^m \, (1+\alpha+\ldots+\alpha^{n-m-1}) \, |g(r_{0},r_{1})|\\
		& \leq   \frac{\alpha^m}{1-\alpha} \,  |g(r_{0},r_{1})|\to 0 ~\quad  \text{as}~ ~ \, m,n\to +\infty.    
	\end{align*}
    
	This shows that the sequence $ \{ r_n \}$ is  $\mathcal{R}$-preserving $g$-Cauchy sequence in $ \Omega $. 
	Owing to the $g-\mathcal{R}$-complete of $ \Omega $, there exists $r^*\in  \Omega $ such that $ \{ r_n \} $ is $g-\mathcal{R}$-convergent to $r^*$.

	Now, we use assumption $(iv)$ to show that $r^*$ is a fixed point of $ S $. As $ \{ r_n \} $ is  $\mathcal{R}$-preserving sequence that is  $g$-$\mathcal{R}$-convergent to $r^*$, then $g-\mathcal{R}$-continuity of $ S$ yields $|g(Sr_n, Sr^*)|\to 0$ as $n\to+\infty.$ 
	Hence, $\{S(r_n)\}$ is  $g-\mathcal{R}$-convergent to $S(r^*)$. But $r_{n+1}= S (r_n)$, is  $g-\mathcal{R}$-convergent to $r^*$. 
	Then owing to the uniqueness of the limit (using Lemma \ref{uniqueness of limit}), we get $S(r^*)=r^*$ and  hence $ r ^* \in F (S ) $.   
	
	Otherwise, suppose that $\mathcal{R}$ is $g$-self-closed.  Again as $\{r_n\}$ is a $\mathcal{R}$-preserving sequence and is $g$-convergent to $x^*$. Then, there exists a subsequence $ \{ r_ {n_l} \} $ of $\{r_n\}$  with $ ( r_{n_l}, r^* ) \in\mathcal{R}$, for all $l\in\mathbb {N}_0$.
    
	On using   $(v)$, we obtain
	\begin{equation*}
	|g( r_{{n_l}+1}, S (r^*)| = |g( S (r_{n_l}), S (r^*))|\leq \alpha |g(r_{{n}_l}, r^*)|\to 0~ ~ \text{as}~ l\to+\infty. 
	\end{equation*}	
    
	Therefore, the sequence $ \{ r_{n_l} \}$ is $g$-$\mathcal{R}$-convergent to $ S (r^*)$. Again, owing to the uniqueness of the limit (using Lemma \ref{uniqueness of limit}), we get $ S (r^*)=r^*$ and hence $r^* \in F ( S ) $.  
\end{proof}
%

\begin{theorem}\label{m2}
	In addition to the assumptions of Theorem \ref{main thm}, if $ S ( \Omega )$ is $g-\mathcal{R}^s$-connected then   $ F ( S )  $ contains atmost one   point.
\end{theorem}
\begin{proof}
	On the lines of the proof of Theorem \ref{main thm}, one can show that $F( S )$ is non-empty. Now, if $F( S )$ is singleton then the proof is obvious. Otherwise, let there exists two distinct elements $ r^*, s ^*\in F(S )$. 
    As $ S (\Omega )$ is $g-\mathcal{R}^s$-connected then there exists a finite path, say $\{p_0,p_1,p_2,\ldots,p_t \}$ from $ r^*$ to $ s^*$ in $ S ( \Omega )$ such that $ r^*=p_0, s^*=p_t $ with $ ( p_i, p_{i+1} ) \in \mathcal{R}$ and $|g(p_i,p_{i+1})|<+\infty$  for each $i\in\{0,1,2,\ldots, t-1\}$.
	Again  $ S $-closedness of $\mathcal{R}$   enable us to write $ ( S^np_i, S^np_{i+1} ) \in \mathcal{R} $ for each $i \in \{0,1,2,\ldots,t-1\}$ and $n\in\mathbb{N}_0.$ Now, 
	\begin{align*}
		|g( r^*, s^*)|=|g( S^np_0, S^np_t)| & \leq  \sum_{i=0}^{t-1} |g( S^np_i, S^np_{i+1})| \\
		& \leq   \alpha \sum_{i=0}^{t-1} |g( S^{n-1}p_i, S^{n-1}p_{i+1})|\\
		& \leq  \alpha^2 \sum_{i=0}^{t-1} |g( S^{n-2}p_i, S^{n-2}p_{i+1})| \\
		  & \vdots \\
		 & \leq   \alpha^n \sum_{i=0}^{t-1} |g(p_i,p_{i+1})| 	\to 0 ~ \quad  \text{as}~ \, n\to+\infty. 
	\end{align*}	
    
Therefore, $ r^*= s^*$ and hence $ F ( S )  $ is a singleton set. 
\end{proof}

Next, we present an example in support of our main result (Theorem \ref{main thm}). 	This example also demonstrates that Theorem \ref{thm of S.Som} fails to exhibit a fixed point for a particular mapping. However for a suitably chosen binary relation $ \mathcal{R} $, the same mapping shows a fixed point.

\begin{example}
 Consider $ \Omega = \mathbb{R}^2 $ with the usual topology and define a function $   g : \Omega \times \Omega \to \mathbb{R} $ by  $ g( ( a_1, v_1 ), ( a_2, v_2 ) ) =    v_1 - v_2 $   for all $ ( a_1, v_1 ), ( a_2, v_2 ) \in \Omega $.  
	Then   
	\begin{enumerate}[(i)]
		\item   $  g $  is continuous on $ \Omega \times \Omega $;
		\item $ | g( a, u )| = | g( u, a )| ~  \forall a, u \in \Omega $;
		\item $  |g( a, u )| \leq |g( a, w )| + |g( w, u )| ~ \forall a, u,   w \in \Omega $.
	\end{enumerate}
    
But the condition $ g ( a, u ) = 0 \implies a = u $ is violated because $ g( ( a_1, u_1 ), ( a_2, u_2 ) ) = u_1 - u_2  = 0 $ can hold whenever $  u_1 = u_2 $ and $ a_1 \neq a_2 $.  

	Now define a self-mapping $ S $ on $ \Omega $ by $ S (  u,  a ) = ( u, \frac{a}{4}) ~ \forall ( u, a ) \in \Omega $. Then 
	$$ | g ( S (( u_1, a_1)), S ( ( u_2, a_2 ) ) ) | = \frac{1}{4} | g( ( u_1, a_1 ), ( u_2, a_2 ) ) | < \frac{1}{2} | g ( ( u_1, a_1 ), ( u_2, a_2 ) ) |.  $$
	Therefore, $ S $ satisfies the topologically Banach contraction condition with respect to $g$ for the contraction constant $ \alpha = \frac{1}{2} $. 
    
	As $ g $ fails to satisfy all the  required properties for all elements of $ \Omega $, the existences of   the fixed point for $ S $ can not be guaranteed by   Theorem \ref{thm of S.Som}. 
    
	Now we define a binary Relation $ \mathcal{R} $ on $ \Omega $ as: $  ( ( a_1, u_1 ), ( a_2, u_2 )) \in    \mathcal{R} \iff  a_1 = a_2 $.
    
	Then $g$ satisfies all the properties of Theorem \ref{main thm} and 
	\begin{enumerate}[(i)]
		\item $  \mathcal{R} $ is  $ S $-closed, since if $ ( ( a_1, u_1 ), ( a_2, u_2 ) ) \in \mathcal{R} $ then $   ( S ( a_1, u_1 ), S ( a_2, u_2 ) ) \in \mathcal{R} $.
		\item $ \Omega $ is $ g-\mathcal{R} $-complete.
		\item For   $ (0, 1) \in \Omega $, $ ( ( 0, 1 ), S (0, 1) ) \in \mathcal{R} $ and hence  $ \Omega ( S; \mathcal{R}) \neq \emptyset $.  
		\item $ S $ is $  g-\mathcal{R} $-continuous.
		\item  $S $ satisfies the contraction condition of  Theorem \ref{main thm}  for $ \alpha  = \frac{1}{2} $.
	\end{enumerate}

	 Thus, the  relaxed conditions of the  Theorem  \ref{main thm} allow $ S $ to have a fixed point. Here which is $ ( 0, 0) $. 
\end{example}

To illustrate the superiority of our newly proposed results over the Theorem \ref{thm of S.Som}, we present an example where the mapping 
$S$ fails to satisfy the contraction condition   by Theorem \ref{thm of S.Som} but fulfills the $g$-$\mathcal{R}$-contraction condition of Theorem \ref{main thm}. Here, one may observe that  the involved binary relation $\mathcal{R}$ allow us for a more flexible contraction condition in the sense that the contraction condition merely satisfy for related elements (related under involved binary relation $\mathcal{R}$) rather than all the elements of the underlying space $\Omega$.

\begin{example}
 Consider $ \Omega = \mathbb{R}^2 $ with the usual topology and define a function $   g : \Omega \times \Omega \to \mathbb{R} $ by  $ g( ( u_1, a_1 ), ( u_2, a_2 )  ) =  | u_1 - u_2 | + | a_1 - a_2 |  ~  \forall ( u_1, a_1 ), ( u_2, a_2 )  \in \Omega $.  
	Then  the following holds:
	\begin{enumerate}[(i)]
		\item   $  g $  is continuous on $ \Omega \times \Omega $;
            \item  $ g ( u, a ) = 0 \implies u = a $;
		\item $ | g( u, a )| = | g( a, u  )| ~  \forall u, a \in \Omega $;
		\item $  |g( u, a )| \leq | g (  u, w )| + | g( w, a )| ~ \forall u, a, w \in \Omega $.
	\end{enumerate}
	Now define a self-mapping $ S $ on $ \Omega $ by $ S ( u, a  ) = ( \frac{u^2}{4}, \frac{ a }{4}) ~ \forall ( u, a ) \in \Omega $. Then we have
	$$ | g ( S ( ( u_1, a_1 ) ), S ( ( u_2, a_2 ) ) ) | = \left| g\left( ( \frac{ u_1^2}{4}, \frac{ a_1}{4}) , ( \frac{u_2 ^2}{4}, \frac{a_2}{4}) \right) \right| = \frac{1}{4} | u_1 ^2 - u_2 ^2 | + \frac{1}{4} |a_1 - a_2 |.
  $$
	The term $ | u_1 ^2 - u_2 ^2 | $ grows quadratically with $ u_1, u_2 \in \mathbb{R} $. So, $ S $ does not satisfy the topologically Banach contraction condition of Theorem \ref{thm of S.Som} with respect to $g$ for any    $ \alpha \in ( 0, 1 ) $.  Therefore,  the existences of   the fixed point for $ S $ can not be guaranteed by   Theorem \ref{thm of S.Som}.  
     
	Now we define a binary Relation $ \mathcal{R} $ on $ \Omega $ as: $  ( ( u_1, a_1 ), ( u_2, a_2 ) ) \in    \mathcal{R} \iff  u_1 = u_2 $.  
	Then   
	\begin{enumerate}[(i)]
		\item $  \mathcal{R} $ is  $ S $-closed, since if $ ( ( u_1, a_1 ), ( u_2, a_2 ) ) \in \mathcal{R} $ then $  ( S ( u_1, a_1 ), S ( u_2, a_2 ) ) \in \mathcal{R} $.
		\item $ \Omega $ is $ g-\mathcal{R} $-complete.
		\item For   $ (0, 1) \in \Omega $, $ ( ( 0, 1 ), S (0, 1) ) \in \mathcal{R} $ and hence  $ \Omega ( S; \mathcal{R}) \neq \emptyset $.  
		\item $ S $ is $  g-\mathcal{R} $-continuous.
	\end{enumerate}
	 
     Moreover  for $    ( ( a_1, y_1 ),  ( a_2, y_2 ) ) \in \mathcal{R} $,   $ S $ satisfies:
     $$ | g ( S  ( a_1, y_1  ), S (  a_2, y_2 )  ) | = \left| g\left( \left( \frac{a_1^2}{4}, \frac{y_1}{4} \right) , \left( \frac{a_2 ^2}{4}, \frac{y_2}{4} \right) \right) \right| =  \frac{1}{4} | y_1 - y_2 | < \frac{1}{2} | g ( ( a_1, y_1 ), ( a_2, y_2 ) ) |.  $$
     Therefore, the contraction condition of  Theorem \ref{main thm} holds for $ \alpha  = \frac{1}{2} $ and hence we can ensure the existence of a unique fixed point of $ S $  by Theorem  \ref{main thm}.
	Clearly,  here the fixed point is $ ( 0, 0) $.  This example shows the genuineness of our newly proved result over the corresponding related results.
	
	\begin{figure*}[ht!]
		\centering
	 	\includegraphics[width=10cm,height=7cm]{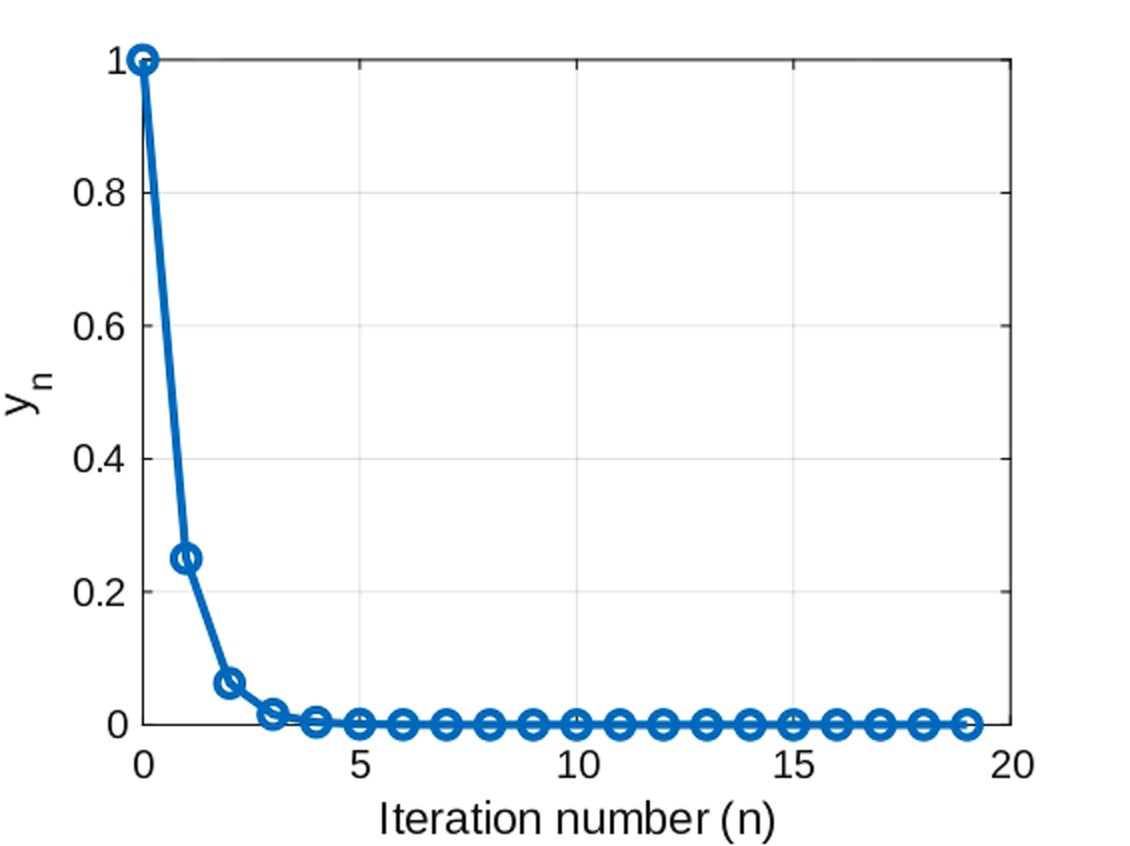}  
	\label{Fig}
		\caption{Convergence of $y$-coordinate under both the map $ T ( x, y ) = ( x, \frac{y}{4} )  $  \&  $ T ( x,   y )  = \left( \frac{x^2}{4}, \frac{y}{4}\right)  $    }
	\end{figure*}

The MATLAB generated graph (Figure \ref{Fig})  displays the evolution of the  $ y $-component,  $ \{ y_n \}  $ of the sequence $ \{ S^n ( x_0, y_0 ) \} = \{ (  0, y_n ) \} $    under both the mapping  $ S (x, y) = (x, \frac{y}{4} ) $ and $ S ( x,   y )  = \left( \frac{x^2}{4}, \frac{y}{4}\right) $, starting from an initial point $ ( 0, 1 ) $. The plot clearly demonstrates geometric convergence of the $ y $-coordinate to zero, which supports the analytical observation that $ S$ satisfies a contraction condition with respect to the function $ g $.    The simulation validates the existence of a fixed point $ (0, 0)   $   within the $g$-$\mathcal{R} $ framework, despite $ g $  not satisfying classical metric properties.  
\end{example}


  In particular, by choosing the binary relation $ \mathcal{R}$ as the universal relation, our main theorem reduces to the framework of Theorem \ref{thm of S.Som} of Som et al. \cite{S.Som}. This is presented as the following corollary.

\begin{corollary} 
	Consider a binary relation $ \mathcal{R}$ defined on   $ \Omega $ and a continuous function $ g : \Omega \times \Omega \to \mathbb{R} $    satisfying the properties (g1)-(g3) of Theorem \ref{main thm} and $ S : \Omega \to \Omega $ be a    topologically $\mathcal{R}$-preserving contraction mapping with respect to $g$.  
	If $ \Omega $ is $g$-complete, then $ F ( S ) $ contains exactly one point. Moreover, for any $ r_0 \in \Omega $,   $ \{ S^n (r_0) \}  $ converges to the unique fixed point of $ S $.
\end{corollary}  
\begin{proof}
	Since $ \mathcal{R} $ is a   universal relation on $  \Omega $, so  $	( u, a ) \in \mathcal{R} $  for all 	 $ u, a \in \Omega $. This implies  $ g ( u, a ) = 0 \implies u = a $,    $ | g ( u, a ) | = | g ( a, u ) | $ and 
	 $  | g ( u, a ) | \leq | g ( u, w ) | + | g ( w, a ) | $ 	holds for    $ u, a, w \in \Omega $ with $ ( u, a ) \in \mathcal{R} $ and $ ( a, w ) \in \mathcal{R} $. 
     
     Moreover, since $\mathcal{R} $ is universal, so $  \Omega ( S; \mathcal{R} ) = \Omega $   is non-empty. Also, under the universal relation $\mathcal{R}$,  $g$-$ \mathcal{R}$-completeness becomes $g$-completeness for the underlying space $\Omega$.  The conditions   $\mathcal{R} $ is  $S $-closed and    either  $ S $ is  $g$-$ \mathcal{R}$-continuous or  $\mathcal{R} $ is  
g-self-closed are trivially satisfied   for the universal relation $\mathcal{R} $. 
Since $ S $ satisfies $ 	|g ( S ( u ), S ( a ) )| \leq \beta |g ( u, a )|, ~ \text{for all} ~ u, a \in \Omega $ and for some $ \beta \in ( 0, 1 ) $  and $\mathcal{R} $ is universal relation, so $ S $ is topologically  $g$-$ \mathcal{R}$-preserving contraction mapping. 

As, $\mathcal{R} $  is the universal relation, so $ ( u, y ) \in \mathcal{R} $  for all 
$ u, y \in \Omega $ and therefore $ \mathcal{R}^s $ relates all points of $ \Omega $.  For any $ u, y \in S ( \Omega ) $,  directly we can take the sequence  $ \{ u, y \} $, and the condition  $  ( u, y ) \in \mathcal{R} $ is satisfied by the universal nature of $ \mathcal{R}$.  Hence, there is always a finite sequence connecting    $ S ( \Omega ) $  $ R^s $-connected. 
Hence from Theorem \ref{main thm} and \ref{m2}, the conclusion follows.
 \end{proof}

\begin{remark}
    The corollary stated above recovers Theorem \ref{thm of S.Som}   as a special case of our main theorem. By choosing $  \mathcal{R} $ as universal relation instead of  binary relation, the $g$-$ \mathcal{R}$-contraction condition coincides with the usual contraction condition of Theorem \ref{thm of S.Som}. This establishes that Theorem \ref{thm of S.Som} of Som et al. \cite{S.Som} is a direct consequence of our generalized result.
\end{remark}


If we consider $g$ to be a metric $d$ on $\Omega$, our main theorem reduces to the framework of Theorem 3.1 of Alam and  Imdad \cite{alam2015relation}. In this context we present the following corollary.

\begin{corollary}
      Let  $\mathcal{R}$ be a binary relation on $ \Omega $ and $ d : \Omega \times \Omega \to \mathbb{R} $  be a   metric.    If $ S : \Omega \to \Omega $ is a mapping satisfying the followings:
	\begin{enumerate}[(i)]
		\item $ \Omega $ is $d-\mathcal{R}$-complete,
		\item $\mathcal{R}$ is $ S $-closed,
		\item $ \Omega ( S; \mathcal{R} ) $ is non-empty,
		\item either $ S $ is ``$d$-$\mathcal{R}$-continuous" or ``$\mathcal{R}$ is $d$-self-closed",
		\item $ S $ is topologically $\mathcal{R}$-preserving contraction with respect to $d$,  
	\end{enumerate} 
	then $ F ( S ) $ contains exactly one point. Moreover, for each $r_0\in \Omega ( S; \mathcal{R})$,   $ \{ S^n(r_0) \} $ converges to the fixed point of $S$. 
\end{corollary}
\begin{proof}
    Since $ g $ is a   metric  on $ \Omega $, so $g $ satisfies all the conditions defined in Theorem \ref{main thm}. Hence the proof follows from the  Theorem \ref{main thm}.
\end{proof}

 \begin{remark}
     The above corollary reduces the Theorem  3.1 of \cite{alam2015relation}   as a special case of our main theorem. By choosing the mapping $ g $ to be metric on $\Omega $, the topological space contains a metric space  structure $ ( \Omega, d ) $. Hence the $ g$-$ \mathcal{R}$-contraction condition coincides with the relational contraction condition of Theorem  3.1 of \cite{alam2015relation}. This shows  that Theorem 3.1 of \cite{alam2015relation} is a direct consequence of our generalized result.
 \end{remark}

\begin{corollary}
  Consider a  natural partial ordered relation $\mathcal{R}:=\preceq$ on   $ \Omega $.  Let $ g : \Omega \times \Omega \to \mathbb{R} $  be a    metric    and $ S : \Omega \to \Omega $ be a mapping for which the following conditions:
	\begin{enumerate}[(i)]
		\item $ \Omega $ is $g$-$\preceq$-complete,
		\item $\preceq$ is $ S $-closed,
		\item $ \Omega ( S; \preceq)$ is non-empty,
		\item either $S $ is ``$g$-$\preceq$-continuous" or ``$\preceq$ is $g$-self-closed",
		\item $ S $ is topologically $\preceq$-preserving contraction with respect to $g$. 
	\end{enumerate} 
    
	holds on $ \Omega$. Then $ F ( S ) $ contains exactly one point.  Moreover, for each $r_0\in \Omega ( S; \mathcal{R})$,   $ \{ S^n(r_0) \} $ converges to the fixed point of $S $.  
\end{corollary}
\begin{proof}
    Since $ g $ is a   metric   on $ \Omega $, so $g $ satisfies the properties defined in Theorem \ref{main thm}. Moreover,    $\mathcal{R}$  being a  natural ordering on $ \Omega $ and hence a binary relation, the conclusion      follows from the  Theorem \ref{main thm}.
\end{proof}

\begin{remark}
    The presence of the metric function $g$ and the natural ordering $ \mathcal{R}$ on $ \Omega $, forms a frame of ordered metric space $ ( \Omega, d, \mathcal{R} ) $. Henceforth   the Corollary  3.3 of Ran and Reurings \cite{ran2004fixed}   can be  presented as a consequence   of our main theorem.   
\end{remark}


\section{Application}
Fractional differential equations   are used as    powerful tools for modeling complex systems characterized by memory and hereditary properties. These equations are  applied across various fields, including physics, biology, engineering and economics to describe phenomena that cannot be captured by classical integer-order models. In this section, we explore the application of fractional differential equations to an economic growth model, highlighting their ability to incorporate non-local effects and memory, which are essential for understanding dynamic systems.

  For this we consider the following equation:
\begin{equation}\label{ode}
		D^\zeta ( f ( t ) ) = h ( t, f ( t ) ), ~~  t \in [ 0, 1 ], ~~ 1 <  \zeta \leq 2
\end{equation}

subject to  the  integral boundary conditions
 \begin{equation}\label{app2}
 	f ( 0 ) = 0, ~ I f ( 1 ) = f' ( 0 ) 
 \end{equation} 
 
 where $ D^\zeta $ represents the Caputo fractional derivative of order $ \zeta $ and defined by
 \begin{equation}\label{app3}
 	D^\zeta ( f ( t ) ) = \dfrac{1}{\Gamma ( i - \zeta )} \int_{0}^{t} ( t- s)^{i-\zeta - 1} f^i ( m ) dm 
 \end{equation} 
 
 such that $ i -1 < \zeta < i, ~ i = [\zeta] + 1$ and $ f : [0, 1 ] \times \mathbb{R} \to [0, \infty ) $ is a continuous function and $ I^\zeta f $ denotes the Reimann-Liouville fractional  integral of order $ \zeta $  of a continuous function $ f : \mathbb{R}^+ \to \mathbb{R} $ given by
 $ I^\zeta f ( t ) = \dfrac{1}{\Gamma ( \zeta )} \int_{0}^{t} ( t- s)^{ \zeta - 1} f  ( m ) dm. $  
The variable $ f(t) $ may represent an economic indicator characterizing the economic health of a region. The nonlinear function $ g(t,f(t) ) $ reflects contributions from various factors, such as innovation, government spending, and the education system, within an economic growth model. The fractional order $ \zeta $ captures the non-local effects and memory inherent in the economic system. The integral boundary conditions $ f(0) = 0 $  and $ I f(1 ) = f'  ( 0 ) $ signify the initiation of economic activity or the observation period and establish a relationship between the accumulated value over the interval $ [0,1] $ and the rate of change of the economic variable at the start of the observation period, for more details see \cite{15,16,17} and citation therein. \\

Let $ \Omega = C [ 0, 1 ] $, set of all continuous function over $ [ 0, 1 ] $. Then the sup metric $ d_\infty ( p_1, p_2 ) = \underset{t \in [ 0, 1 ]}{\sup} | p_1 ( t ) -  p_2 (t ) | $ induces the usual topology on $ \Omega $.

Now we define a binary relation $ \mathcal{R} $ on $ \Omega \times \Omega $ by $ ( p_1, p_2 ) \in \mathcal{R} \iff p_1 ( t ) \leq p_2 ( t ) ~ \forall t \in [ 0, 1 ] $.
Next consider a mapping $   g : \Omega \times \Omega \to \mathbb{R} $ by  $ g( q_1, q_2 ) = \underset{t \in [ 0, I ]}{\sup}  (   q_1 ( t ) - q_2 ( t ) ) ~  \forall q_1, q_2 \in \Omega $.  
Then  
\begin{enumerate}[(i)]
	\item $ g $ satisfies the properties (g1)-(g3) of Theorem \ref{main thm}.  
	\item   $  g $  is continuous on $ \Omega \times \Omega $.
\end{enumerate}
~\\
In this context, we state the following theorem. 

\begin{theorem}
	Consider the non-linear fractional differential equation (\ref{ode}) and suppose the function $h$ satisfies the following conditions:
	\begin{enumerate}[(i)]
		\item $ h $ is a non-decreasing function on the second variable;
		\item for each $ \mu \in [ 0, 1 ] $ and $ ( u, v ) \in \mathcal{R} $, $ h $ satisfies 
		$$ | h ( \mu, u ( \mu ) ) - h ( \mu, v ( \mu ) ) | \leq \frac{\alpha \Gamma ( \alpha + 1 ) }{4} | u ( \mu ) - v( \mu ) | ~ \text{where} ~ \alpha \in ( 0, 1 ).  $$
	\end{enumerate}
Then   the fractional differential equation (\ref{ode}) admits a solution in $C[0,1] $.
\end{theorem}
\begin{proof}
	We recall the topological space $ \Omega $,   mapping $g$ and binary relation $ \mathcal{R} $ defined above. Define a mapping $ T : \Omega \to \Omega $ by 
	\begin{equation}\label{int equn}
		T ( u( r ) ) = \dfrac{1}{\Gamma ( \zeta )} \int_{0}^{r} ( r- s)^{ \zeta - 1} h (s, u (s)) ds +   \dfrac{2 r }{\Gamma ( \zeta )} \int_{0}^{1}  \left( \int_{0}^{s} ( s - m)^{ \zeta - 1} h  ( m, u ( m ) ) dm    \right) ds~~ \qquad ~ \forall u \in \Omega.
	\end{equation}
Clearly the solution of (\ref{ode}) is a fixed point of $T$ in $ \Omega $. \\
Now with respect to this $T$, we verify the conditions of the Theorem \ref{main thm}. \\
Observe  that, for all $u, v \in \Omega $ with $ ( u, v ) \in \mathcal{R} $ and $ r \in [ 0, 1 ] $, 
\begin{align*}
	T ( u ( r ) ) = &  \dfrac{1}{\Gamma ( \zeta )} \int_{0}^{r} ( t- s)^{ \zeta - 1} h (s, u (s)) ds +   \dfrac{2 r }{\Gamma ( \zeta )} \int_{0}^{1}  \left( \int_{0}^{s} ( s - p )^{ \zeta - 1} h  ( p, u ( p ) ) dp    \right) ds \\
	\leq & \dfrac{1}{\Gamma ( \zeta )} \int_{0}^{r} ( t- s)^{ \zeta - 1} h (s, v (s)) ds +   \dfrac{2 r }{\Gamma ( \zeta )} \int_{0}^{1}  \left( \int_{0}^{s} ( s - p)^{ \zeta - 1} h  ( p, v ( p ) ) dp    \right) ds \\
    = &  T ( v ( r ) ).
\end{align*}

Hence $ ( u, v ) \in \mathcal{R} \implies   ( Tu, Tv ) \in \mathcal{R} $ and therefore $ \mathcal{R}$ is $T$-closed. 

Next consider a $g$-Cauchy sequence $ \{ f_n\} $ in $ \Omega $. Therefore
$$  \underset{m,n \to \infty }{\lim} | g ( f_m, f_n ) | = 0  ~\qquad ~ or ~~
	  \underset{m,n \to \infty }{\lim} \underset{t \in [0, I] }{\sup} ~ |  f_m ( t ) - f_n ( t ) | = 0.  $$

This is the pointwise convergence of $ \{ f_n\} $ in $X$ with respect to the   metric $ d_\infty$. As the metric space $  ( \Omega,   d_\infty ) $ is complete, so $ \{ f_n\} $ converges to some $ f$ in $ ( \Omega,   d_\infty ) $. Hence $  f_n \to f $ (convergence in $ g $-sense). Moreover, if $ \{ f_n\} $ is $ \mathcal{R}$-preserving then we must have $( f_n, f)\in\mathcal{R}$ for each $n\in\mathbb{N}_0$. Thus $ \Omega $ is  $g-\mathcal{R}$-complete. 

As $ C [ 0, 1 ] $ is non-empty, we can take   a function $ f_0 \in C [ 0, 1 ] $. Then   $ T f _0 \in C [ 0, 1 ] $. If there exist some $ r > 0 $ such that $ g ( f_0, T f _0 ) \leq r $ then $ ( f_0, T f _0 ) \in \mathcal{R} $. Otherwise choose $ f_1 \in C [ 0, 1 ] $ ($e.g.$, a function close enough to $f \in  C [ 0, 1 ] $) such that the condition satisfied. The richness of $  C [ 0, 1 ] $ ensures that such $f$ exists.  This way we can conclude $ \Omega (T; \mathcal{R} ) $ must be non-empty.    

Suppose $ \{ f_n\} $ is a $ \mathcal{R} $-preserving sequence in $ \Omega$ which converges to $ f \in \Omega $. Then 
\begin{align*} 
	& \underset{ n \to \infty }{\lim} ~ T ( f_n ) ( t ) \\
    = & \underset{ n \to \infty }{\lim} ~ \dfrac{1}{\Gamma ( \zeta )} \int_{0}^{t} ( t- s)^{ \zeta - 1} h (s, f_n (s)) ds +   \dfrac{2 t }{\Gamma ( \zeta )} \int_{0}^{1}  \left( \int_{0}^{s} ( s - m)^{ \zeta - 1} h  ( m, f_n ( m ) ) dm    \right) ds \\
	 = & \dfrac{1}{\Gamma ( \zeta )} \int_{0}^{t} ( t- s)^{ \zeta - 1} h (s, f (s)) ds +   \dfrac{2 t }{\Gamma ( \zeta )} \int_{0}^{1}  \left( \int_{0}^{s} ( s - m)^{ \zeta - 1} h  ( m, f ( m ) ) dm    \right) ds \\
     = & T ( f ) ( t )  \quad \forall t \in [ 0, 1 ]. 
\end{align*} 

Thus $T$ is $g-\mathcal{R}$-continuous.

Next consider $ u, v \in \Omega $ with $ ( u, v ) \in \mathcal{R} $. Then 
$$ | g ( T u, T v ) | =  \underset{t \in [0, 1 ] }{\sup} ~ | T u ( t ) -  T v ( t ) |   $$
and  for each $ t \in [ 0, 1  ] $
\begin{align*}
	&  | T u ( t ) -  T v ( t ) |  \\
	= & ~  | \dfrac{1}{\Gamma ( \zeta )}\int_{0}^{t} ( t- s)^{ \zeta - 1} \left[  h (s, u (s) ) -  h (s, v (s) )  \right] ds + \\
		& \hskip 80 pt  \dfrac{2 t }{\Gamma ( \zeta )} \int_{0}^{1} \left(  \int_{0}^{s} ( s - p)^{ \zeta - 1} [ h  ( p, u ( p ) ) -  h  ( p, v ( p ) ) ] dp \right) ds   |   \\
		\leq & ~ \dfrac{1}{\Gamma ( \zeta )} \int_{0}^{t} ( t- s)^{ \zeta - 1}  |  h (s, u (s) ) -  h (s, v (s) )  | ds + \\
		&  \hskip 80 pt  \dfrac{2 t }{\Gamma ( \zeta )} \int_{0}^{1} \left(  \int_{0}^{s} ( s - p)^{ \zeta - 1} | h  ( p, u ( p ) ) -  h  ( p, v ( p ) )  | dp   \right) ds \\
		\leq & ~  \dfrac{1}{\Gamma ( \zeta )} \int_{0}^{t} ( t- s)^{ \zeta - 1}   ~ \dfrac{\alpha \Gamma ( \zeta + 1 ) }{4} | u ( s ) - v ( s ) | ds + \\
		&  \hskip 80 pt  \dfrac{2 t }{\Gamma ( \zeta )} \int_{0}^{1} \left(  \int_{0}^{s} ( s - p)^{ \zeta - 1} ~ \dfrac{\alpha  \Gamma ( \zeta + 1 ) }{4}  |   u ( p )   -    v ( p )    | dp   \right) ds \\
		= & ~ \dfrac{ \alpha  \Gamma ( \zeta + 1 ) }{ 4 \Gamma ( \zeta )} \left[  \int_{0}^{t} ( t- s)^{ \zeta - 1}  | u ( s ) - v ( s ) | ds  + 2 t \int_{0}^{1} \left(  \int_{0}^{s} ( s - p )^{ \zeta - 1}    |   u ( p )   -    v ( p )    | dp   \right) ds \right]  \\
		\leq & ~ \dfrac{ \alpha    \zeta   }{ 4  }  \underset{t \in [ 0, 1]}{\sup}  |   u ( t )   -    v ( t ) | \left( \int_{0}^{t} ( t- s)^{ \zeta - 1} ds + 2t  \int_{0}^{1} \left(  \int_{0}^{s} ( s - p )^{ \zeta - 1}  dp   \right) ds \right)  \\
		\leq & ~ \dfrac{ \alpha  \zeta  }{ 4  } ~ | g ( u, v )|  ~ \left( \dfrac{1 + 2 t}{\zeta}   \right)  
		<   ~ \alpha | g ( u, v )|. 
\end{align*}

Henceforth,
$$   | g ( T u, T v ) | \leq \alpha | g ( u, v ) |.  $$ 

Therefore $T$ satisfies the conditions of the  Theorem \ref{main thm} and hence $T$ has a fixed point in $ \Omega $.  Consequently   the fractional differential equation \ref{ode} admits a solution in $ C [ 0, 1 ] $. 
\end{proof}

\subsection{Numerical Illustration of the Application}

\hspace{0.2cm}

To demonstrate the practical implementation and validate the theoretical results established in the previous section, we present a numerical example based on the iterative scheme derived in the application. Specifically, we consider the metric space   $ (\Omega , d_\infty )  $,   binary relation $ \mathcal{R} $ on $ \Omega \times \Omega $  and  $   g $ defined above. In particular we take $ h( t, u(t ) ) =  \frac{1}{16} u (t ) + \sin (t ), ~ \forall t \in [ 0, 1 ] $. Then \\
$\bullet $ $  h $ is     non-decreasing with respect to  $ u $; \\
$\bullet $  for each $ \mu \in [ 0, 1 ] $ and $ ( u, v ) \in \mathcal{R} $, $ h $ satisfies 
$$ | h ( \mu, u ( \mu ) ) - h ( \mu, v ( \mu ) ) | \leq \frac{1}{16} \, | u ( \mu ) - v( \mu ) | <  \frac{\alpha \Gamma ( \alpha + 1 ) }{4} \, | u ( \mu ) - v( \mu ) |  $$

for $ \alpha  = \frac{1}{2} $.  

Next we consider $ \zeta = 0.9 $ and initial guess $ u_0 ( t ) = 0 $. We now compute the operator value $T$ defined in equation (\ref{int equn})   for $ u = u_0, \, u_1, \cdots $. 

For, we define the iterative sequence 
$$ u_{n+1} ( t )= T u_n $$

and we compute $ u_n ( t ) $ for several iterations until the sequence converges.

We numerically approximate the integrals in the operator
\begin{align*}
	u_{n+1} ( t )  = & \dfrac{1}{\Gamma ( 0.9 )} \, \int_{0}^{r} ( r - s)^{  - 0.1} \, \left\{ \frac{1}{16} \, u_n (s )  + \sin (s ) \right\} \,  ds + \\
	& \hskip 50 pt 	\dfrac{2 t }{\Gamma ( 0.9 )} \, \int_{0}^{1}  \left( \int_{0}^{s} ( s - m)^{  - 0.1}  \, \left\{ \frac{1}{16} \, u_n (m ) + \sin (m ) \right\}  \,  dm    \right) ds
\end{align*} 	 

starting from the zero function $ u_0  (t)=0 $. The iteration   is continued for a fixed number of steps  and the sequence \( \{ \| u_{n+1} - u_n \|_\infty \} \) is monitored in MATLAB  to analyze convergence.  

\begin{figure*}
	[ht!]
	\centering
	\includegraphics[width=10cm,height=7cm]{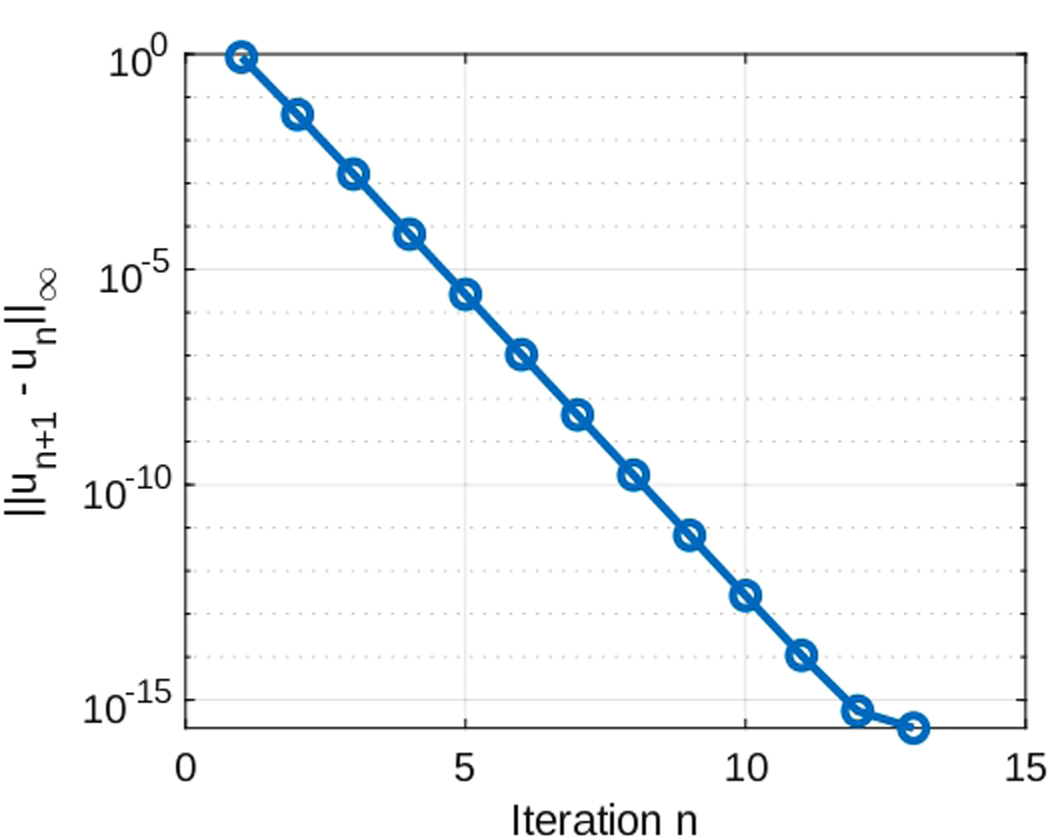}  
	\label{Fig-App} 
	\caption{Convergence profile of the iterative scheme:  $ \| u_{n+1} - u_n \|_\infty $ vs the iteration number $ n $     }
\end{figure*}

\textbf{Convergence Analysis of the Iterative Scheme}:
To validate the theoretical claim   and observe the contractive behavior of the proposed iterative operator, we compute the sequence $ \{ \| u_{n+1} - u_n \|_\infty \}  $ at each iteration using MATLAB. Above Figure \ref{Fig-App} presents the convergence profile, showing the sup norm error $ \{ \| u_{n+1} - u_n \|_\infty \}  $ vs the iteration number $ n $. 
The error decreases from $ \mathcal{O} (1) $  to below machine precision $ \mathcal{O} (10^{-15} ) $  in fewer than 15 iterations. The nearly linear decay in the semilogarithmic scale confirms that the operator satisfies a contractive condition  and ensures rapid convergence to the unique solution. 

This numerical experiment supports the theoretical fixed-point result and illustrates the practical feasibility of applying the proposed iterative method to fractional-type integral equations.

\section{Conclusion}
In this article, we introduced the concept of topologically $ \mathcal{R} $-preserving BCP   on topological spaces that combines the notions of binary relations and continuous functions to extend classical metric fixed-point results. These results provide sufficient conditions for the existence and uniqueness of fixed points.
 This generalization addresses scenarios where existing theorems, such as Theorem \ref{thm of S.Som} of \cite{S.Som} may fail to guarantee the existence of fixed points but by incorporating a suitable binary relation $ \mathcal{R} $, our framework   allows us to grantee the existence of fixed point. We used MATLAB for effective visualization of the  convergence behavior, highlighting how relational and non-metric fixed point frameworks can be demonstrated computationally. An application is provided that highlighted  the applicability of the newly obtained results.  
  
 As a continuation of our work, it directions for exploring further generalizations can be: relaxing the conditions on the binary relation $  \mathcal{R} $  or the function $  g $, and investigating  to the specific problems in optimization, dynamic systems, and applied sciences that remains an intriguing area of research.\\

\noindent {\bf Acknowledgements} All authors are thankful to the learned referees for sparing their valuable time to review this manuscript.\\

\noindent {\bf Author contributions} All authors contributed equally to this manuscript.\\

\noindent {\bf Data availability} The authors confirm that there is no associated data.\\

\noindent {\bf \large Declarations}\\

\noindent {\bf Funding declaration} No funding.\\

\noindent {\bf Competing interests} The authors declare no competing interests.

\end{document}